\documentclass[12pt]{amsart}
\usepackage{graphicx} 
\usepackage{amsmath,amssymb,enumerate,xcolor, enumitem}

\usepackage[breaklinks=true,colorlinks=true,linkcolor=blue,citecolor=red,urlcolor=blue,psdextra,pdfencoding=auto]{hyperref} 
\allowdisplaybreaks

\newtheorem{theorem}{Theorem}[section]
\newtheorem{remark}{Remark}[section]

\newtheorem{corollary}[theorem]{Corollary}
\newtheorem{lemma}[theorem]{Lemma}

\newtheorem{fact}{Fact}

\newcommand{\weg}[1]{}

\renewcommand \a{\alpha}

\newcommand \la{\lambda}

\newcommand \br{\mathbb{R}}

\newcommand \Span{\operatorname{Span}}

\newcommand \cH{\mathcal{H}}

\newcommand \sK{\mathsf{K}}

\newcommand \sL{\mathsf{L}}

\usepackage[backend=biber]{biblatex} 
\title{ Killing tensors on reducible spaces}

\author{Vladimir S. Matveev}
\address{Fakult\"{a}t f\"{u}r Mathematik und Informatik, Friedrich-Schiller-Universit\"{a}t, 07737 Jena, Germany}
\email{vladimir.matveev@uni-jena.de}
\thanks{The first named author was partially supported by ARC Discovery Grant DP210100951 and by the DFG projects 455806247 and 529233771. The first named author is thankful to La Trobe University for hospitality.}

\author{Yuri Nikolayevsky}
\address{Department of Mathematical and Physical Sciences, La Trobe University, VIC 3086, Australia}
\email{Y.Nikolayevsky@latrobe.edu.au}
\thanks{The second named author was partially supported by ARC Discovery Grant DP210100951. The second named author is thankful to Friedrich-Schiller-Universit\"{a}t for hospitality.} %


\begin{document}

\begin{abstract}
We prove that on the product of two Riemannian manifolds one of which is compact, any Killing tensor is reducible, that is, is the sum of products of Killing tensors on the factors. The same is true for the lifts to the universal cover of Killing tensors on a compact manifold with reducible holonomy. We give a local description of Killing tensors on product manifolds and present an example of a complete product manifold whose factors are locally irreducible which admits an irreducible Killing tensor field.
\end{abstract}


\maketitle

\section{Introduction} \label{s:intro}

Let $(M,g) $ be a connected Riemannian manifold. We consider its geodesic flow which we view as the Hamiltonian system on the cotangent bundle $T^*M$ endowed with the standard Poisson structure $\{ \ , \ \}$,
generated by the Hamiltonian $\cH= \tfrac{1}{2}g^{ij}p_ip_j$.
We say that a function $F: T^*M \to \mathbb{R}$ is \emph{an integral} for the geodesic flow of $g$, if $\{ \cH, F\}=0$. Geometrically, the condition $\{ \cH, F\}=0$ means that the function $F$ is constant on every trajectory of the Hamiltonian system.
We say that the integral $F$ is \emph{polynomial in momenta}, if the restriction of $F$ to every cotangent space $T_x^*M$
is a polynomial (of course, the coefficients of the polynomial $F$ may depend on the position $x$). It is well known and is easy to see that if $F=F_0+ F_1 +...+ F_d $ is an integral polynomial in momenta, where $F_i$ denotes the homogeneous polynomial in momenta of degree $i$, then every $F_i, \, i=0, \dots, d$, is also an integral.

In differential-geometric language, homogeneous polynomial in momenta integrals are essentially the same as Killing tensors. Namely, for a function $F_d$ which is a homogeneous polynomial of degree $d$ in momenta, define the symmetric $(0,d)$-tensor $K_{i_1\dots i_d}$ such that $F_d= K^{i_1\dots i_d}p_{i_1}\dots p_{i_d}$ (where here and below we use the metric $g$ for raising and lowering the indices).
It is known that $F_d$ is an integral if and only if $K$ is \emph{a Killing tensor}, that is, if it satisfies \emph{the Killing equation}
\begin{equation*}
 K_{(i_1\dots i_d,j)}=0,
\end{equation*}
where the comma denotes the covariant differentiation with respect to the Levi-Civita connection of $g$, and the parentheses denote the symmetrisation by all the indices. 

The main results of our paper are Theorems~\ref{thm:2} and \ref{thm:3} which we state now, and Theorem~\ref{t:decompo} for which we need a little preparation. For a Riemannian manifold $(M,g)$, denote $\sK^d(M), \, d \ge 0$, the space of integrals which are homogeneous polynomial in momenta of degree $d$. This space is finite dimensional (for a stronger fact see Lemma~\ref{l:findim} below).

\begin{theorem} \label{thm:2}
 Let $(M_1, g_1)$ be a compact connected Riemannian manifold of class $C^2$ of dimension $m_1 \ge 1$, and let $(M_2, g_2)$ be a connected \emph{(}not necessarily complete\emph{)} Riemannian manifold of class $C^2$ of dimension $m_2 \ge 1$. We consider the Riemannian product $(\overline{M}, \overline{g})=(M_1 \times M_2, g_1 +g_2)$.

 Then any $F \in \sK^d(\overline{M})$ is a finite sum of the products of the form
 \begin{equation} \label{eq:1}
 F_1^{d-\ell} F_2^{\ell} \ ,
 \end{equation}
 where $F_1^{d-\ell} \in \sK^{d-\ell}(M_1)$ and $F_2^{\ell} \in \sK^{\ell}(M_2)$ for some $0 \le \ell \le d$.
\end{theorem}

\begin{theorem} \label{thm:3}
 Let $(M, g)$ be a compact, connected Riemannian manifold of class $C^2$ of dimension $m \ge 2$ whose holonomy group is reducible. We isometrically identify the universal cover $\overline{M}$ of $M$ with the Riemannian product of two \emph{(}complete, connected\emph{)} Riemannian manifolds which we denote $(M_1, g_1)$ and $(M_2, g_2)$, respectively.

 Then the lift of any $F \in \sK^d(M)$ is a finite sum of the products of the form~\eqref{eq:1}.
\end{theorem}
Note that the decomposition of the universal cover $\overline{M}$ into the Riemannian product $M_1\times  M_2$ in Theorem~\ref{thm:3} may not be unique. Indeed, the universal cover may be isometric to the Riemannian product of three or more manifolds which can be differently combined into $M_1$ and $M_2$. In addition, one or both components can be flat or may contain flat factors. Theorem \ref{thm:3} tells us that for any such decomposition of $\overline{M}$, any polynomial in momenta integral is reducible relative to this composition.

For $d=1$, Theorem~\ref{thm:3} follows from \cite[Corollary p.~285]{Tachibana}, and moreover, its local version is also true. The case $d=2$ for the Riemannian product of two compact manifolds follows from \cite[Theorem 5.1]{Heil}.

Note that if $F_1$ is an integral for $g_1$ and $F_2$ is an integral for $g_2$ then their product $F_1 F_2$ is trivially an integral for $\overline{g}$. Theorems~\ref{thm:2} and \ref{thm:3} tell us that under minor global assumptions, every polynomial in momenta integral of the Riemannian product is reducible, in the sense that it can be represented as the sum of the products of the integrals of the form~\eqref{eq:1} on the factors. Without imposing these assumptions, the structure of the integrals on the product manifold becomes more complicated --- we present the corresponding general result in Theorem~\ref{t:decompo} below. In particular, the integrals on the product may not be reducible in the above sense: in Section~\ref{ex:noncomp}, we give an example of two complete, locally irreducible Riemannian manifolds whose Riemannian product admits a polynomial in momenta integral which can not be decomposed into the sum of the products of the form~\eqref{eq:1} on the factors. 

To state Theorem~\ref{t:decompo} we need to introduce some notation. For simplicity, in the context of Theorem~\ref{t:decompo} we assume all the objects to be $C^\infty$ smooth (although $C^n$ with $n$ large enough is sufficient). For a $C^\infty$ function $f:T^*M\to \mathbb{R}$ and for $k\in \mathbb{N}$ we denote by $H^kf \in C^\infty(T^*M)$ the function given by
\begin{equation}\label{eq:H}
H^kf:= \{\{\dots\{\underbrace{ \cH, \{ \cH, \dots \{ \cH}_{\textrm{$k$ times}}, f\}\}\}\dots \}\}.
\end{equation}
We denote by $\sL_k^d(M)$ the space of functions $f \in C^\infty(T^*M)$ each of which is either zero or is a homogeneous polynomial in momenta of degree $d \ge 0$ which satisfies $H^kf=0$. Note that $\sL_1^d(M) = \sK^d(M)$. The elements of $\sL_k^d(M)$ have the following geometric characterisation (see Fact~\ref{f:fact1}\eqref{it:dergeod}): the restriction of any function $f\in \sL_k^d(M)$ to any trajectory $(\gamma(s),p(s))$ of the geodesic flow is a polynomial of degree at most $k-1$ of the time (of the arc length parameter) $s$. In Lemma~\ref{l:findim} we will show that each space $\sL_k^d$ is finite dimensional. \weg{ Indeed, the space $H^1_d(M)$ coincides with the space of polynomial in momenta integrals of degree $d$ and is finite-dimensional by \cite[Theorem 4.3]{Thompson}. Next, on the space $\sL^2_d(M)$ consider the mapping $f \mapsto H^1f$. The kernel of this mapping consists of integrals, polynomial in momenta of degree $d$, and is finite-dimensional by \cite[Theorem 4.3]{Thompson}. The image of this mapping consists of integrals, polynomial in momenta of degree $d+1$, and is again finite-dimensional by \cite[Theorem 4.3]{Thompson}. This implies that $\sL^2_d(M)$ is finite dimensional. Repeating this argumentation we can inductively show that $\sL^3_d(M)$, $\sL^4_d(M)$, \dots , are finite-dimensional. }

For Riemannian manifolds $(M_1,g_1)$ and $(M_2, g_2)$ we denote $(\overline{M},\overline{g}) = (M_1 \times M_2, g_1+g_2)$ their Riemannian product. For $k \ge 1$ and $s_1, s_2 \ge 0$, let $f_i \in \sL_k^{s_i}(M_i), \, i=1,2$. One can easily check that the function $f \in C^\infty(T^*\overline{M})$ defined by
\begin{equation}\label{eq:fM1M2}
 f=\sum_{\ell=0}^{k-1} (-1)^\ell (H_1^\ell f_1)(H_2^{k-\ell-1}f_2)
\end{equation}
is a polynomial integral of degree $s_1+s_2+k-1$ of the geodesic flow of $\overline{g}$. Moreover, if $k > 1$ and if we choose $f_1 \in \sL_k^{s_1}(M_1) \setminus \sL_{k-1}^{s_1}(M_1)$ and $f_2 \in \sL_k^{s_2}(M_2) \setminus \sL_{k-1}^{s_2}(M_2)$, then the resulting function $f$ is not the sum of the products of the polynomial integrals on the factors. Indeed, if it were so we would have had $H_1f=0$, and so $\sum_{\ell=0}^{k-1} (-1)^\ell ((H_1)^{\ell+1} f_1)(H_2^{k-\ell-1}f_2)=0$. But the terms of this summation have different degrees in the momenta on $M_1$, and so each of them must be zero, which contradicts the choice of $f_1$ and $f_2$.

We prove the following.

\begin{theorem} \label{t:decompo} 
 Let $(M_1,g_1)$ and $(M_2,g_2)$ be connected, pseudo-Rieman\-nian manifolds \emph{(}of arbitrary signatures\emph{)} of class $C^\infty$, and let $(\overline{M},\overline{g})=(M_1 \times M_2,g_1+g_2)$ be their product. In the above notation, any $\overline{F} \in K^d(\overline{M})$ is a finite sum of the integrals of the form~\eqref{eq:fM1M2} with $s_1+s_2+k-1=d$.
\end{theorem}

Our interest to the study of the (quite natural) questions on the structure of Killing tensors on the product manifolds stems from our project on Killing tensors on symmetric spaces motivated by \cite[Question 3.9]{BMMT}, see also \cite{MN}. The long term goal is to describe Killing tensors on all symmetric spaces, and Theorem~\ref{thm:2} may help one to restrict the study to the case of irreducible symmetric spaces.

The authors would like to thank A.~Moroianu, G.~Rastelli and U.~Semmelmann for their useful comments, and the anonymous referee for valuable suggestions.

\section{Proof of Theorem ~\ref{thm:2} and Theorem~\ref{thm:3}}

We start with two technical lemmas, Lemma~\ref{lem:1} and Lemma~\ref{lem:2}. Corollary~\ref{cor:2} of the Lemma~\ref{lem:2} will immediately imply both Theorems~\ref{thm:2} and~\ref{thm:3}.

\begin{lemma} \label{lem:1}
We consider $(\mathbb{M}^2, g_{euclidean})=(\mathbb{R}, dt_1^2) \times(\mathbb{I}, dt_2^2), $ where $\mathbb{I}\subseteq \mathbb{R}$ is an open interval and $t_1$ and $t_2$ are the Cartesian coordinates on $\mathbb{R}$ and on $\mathbb{I}$, respectively. Suppose that the function $F(t_1, t_2, p_1, p_2) = \sum_{i=0}^d a_i(t_1,t_2) p_1^{d-i}p_2^i$ is an integral on $\mathbb{M}^2$ which has the property that for any fixed $p=(p_1, p_2)$ and any $t_2 \in \mathbb{I}$, the function $t_1 \mapsto |F(t_1, t_2, p_1, p_2)|$ is bounded. Then the coefficients $a_i$ are constant.
 \end{lemma}

\begin{proof} It is known (e.g., \cite[Theorem 4.5]{Thompson}, \cite[Theorem 4.4]{Takeuchi} or \cite{Milson}) that for any homogeneous polynomial in momenta integral on $(\mathbb{R}^2, g_{euclidean})$, there exists a homogeneous polynomial $P( \ \cdot \ , \ \cdot \ , \ \cdot )$ (in three variables, with constant coefficients) such that $F=P(p_1, p_2,t_1 p_2 - t_2 p_1) $. Since for any fixed $p_1, p_2, t_2$, the function $t_1\mapsto |P(p_1, p_2,t_1 p_2 - t_2 p_1)| $ is bounded, the polynomial $P$ does not depend on the third variable implying that all the coefficients $a_i$ are constant, as claimed.
\end{proof}

We use the following notation in Lemma~\ref{lem:2} below. Consider the Riemannian product $(\overline{M}, \overline{g})= (M, g)\times (\mathbb{I}, dt^2)$, where $(M,g)$ is a Riemannian manifold of dimension $m$ and $\mathbb{I}\subseteq \mathbb{R}$ is an open interval. Denote by $t$ the standard coordinate on $\mathbb{I}$ and by $(t,q)$ the corresponding coordinates on $T^*\mathbb{I}$. Let $x=(x_1,...,x_m)$ be a local coordinate system on $M$ and denote the corresponding coordinates on $T^*M$ by $(x,p)$. Let $\overline{\cH}$ be the Hamiltonian of the geodesic flow of $\overline{g}$ and let $U_{(x,t) }\overline{M}:= \{(p,q)\in T^*_{(x,t) }\overline{M} \mid \overline{\cH}(x,t, p,q)=\frac12\}$ denote the unit sphere in $T^*_{(x,t)}\overline{M}$.

\begin{lemma} \label{lem:2}
 Let $(M,g)$ be a connected, complete Riemannian manifold, and let $(\overline{M}, \overline{g})= (M, g)\times (\mathbb{I}, dt^2)$. Adopt the notation above. Suppose that $\overline{F} \in \sK^d(\overline{M})$ is such that for any $t \in \mathbb{R}$, the function $(x\in M, (p,q)\in U_{(x,t) }\overline{M} ) \mapsto |\overline{F}(x,t, p, q)|$ is bounded.

 Then $\overline{F} (x, t, p, q)= \sum_{\ell=0}^d q^\ell F^{d-\ell}(x,t)$, where $F^{d-\ell} \in \sK^{d-\ell}(M)$.
\end{lemma}

\begin{proof}
Given $\overline{F} \in \sK^d(\overline{M})$, we decompose it as
\begin{equation*}
\overline{F} = \sum_{\ell=0}^d q^\ell F^{d-\ell},
\end{equation*}
where $F^{d-\ell}=F^{d-\ell}(x,t,p)$ is a polynomial of degree $d-\ell$ in $p$ whose coefficients may depend on $x$ and $t$.

We first show that the functions $F^{d-\ell}(x,t, p),\, \ell=0, \dots, d$, do not depend on $t$, so that $F^{d-\ell}=F^{d-\ell}(x, p)$. Consider a geodesic $\gamma:\mathbb{R}\to M$ parameterised by an arc length $t_1 \in \br$. Next, consider the totally geodesic immersion $T:\mathbb{R} \times \mathbb{I} \to M\times \mathbb{I}$ given by $(t_1, t) \mapsto (\gamma(t_1), t) $. The pullback of the integral $\overline{F}$ under $T$ satisfies the assumptions of Lemma~\ref{lem:1} implying that its coefficients $F^{d- \ell}$ do not depend on $t$, as we claimed.

Let us now show that all the functions $F^{d-\ell}(x,p)$ are integrals of the geodesic flow of $ {g}$. Note that the Hamiltonian $\overline{\cH}$ of the geodesic flow of the product metric $\overline{g}$ is given by
\begin{equation*}
\overline{\cH}(x,t, p, q)= \cH(x,p)+ \tfrac{1}{2}q^2,
\end{equation*}
where $\cH=\cH(x,p)$ is the Hamiltonian of $g$, and that
\begin{equation*}
\{ F^{d-\ell}(x,p), q\}= \{q^\ell, q\}=0 .
\end{equation*}
Using this, we obtain
\begin{equation*}
0=\{\overline{\cH}, \sum_{\ell=0}^d q^\ell F^{d-\ell}\} = \sum_{\ell=0}^d q^\ell \{\cH, F^{d-\ell}\}
\end{equation*}
implying that $\{\cH, F^{d-\ell}\} =0$, as we want.
\end{proof}

\weg{
\begin{remark} \label{rem:1} Later, in the proof of Corollary~\ref{cor:2}
 we will use Lemma \ref{lem:2} in the following equivalent form: (under assumptions of Lemma \ref{lem:2}), for any $g$-geodesic
 $ \overline{\gamma}(t) =(\gamma(t), \gamma_2(t)) $ (of course $\gamma_2(t)$ is 1-dimensional and is simply $t\mapsto \textrm{const}\ t$) there exists polynomial in velocities integrals $F^{0}(x,u), F^{1}(x,u), \dots , F^{d}(x,u)$ of the geodesic flow of $g$ such that
 \begin{equation} \label{eq:6}
 \overline{ F}\left( x, \gamma_2(t), p,q \right)= \sum_{\ell=0}^d \left(\overline{g}(\dot\gamma_2(t), p ) \right)^\ell F^{d-\ell}(x,q),
 \end{equation}
where $F^{d-\ell}$ are polynomial in velocities integrals of the geodesic flow of $g$.

\end{remark}}

\begin{corollary}\label{cor:2}
 Consider the Riemannian product $(\overline{M}, \overline{g})= (M_1, g_1)\times (M_2, g_2)$ of two connected Riemannian manifolds such that $(M_1, g_1)$ is complete. Let $\overline{F} \in \sK^d(\overline{M})$ be such that for any $y \in M_2$, the function $(x\in M_1, \overline{p}\in U_{(x,y) }\overline{M} ) \mapsto |\overline{F}(x,y,\overline{p})| $ is bounded. Then $\overline{F}$ is a finite sum of the products of the form $F_1^{d-\ell} F_2^{\ell}$, where $F_1^{d-\ell} \in \sK^{d-\ell}(M_1)$ and $F_2^{\ell} \in \sK^\ell(M_2)$.
\end{corollary}
\begin{proof}
Let $x =(x_1,..., x_{m_1}),\, \dim M_1=m_1$, be local coordinates on $M_1$ and denote by $(x,p)$ the corresponding coordinates on the cotangent bundle of $M_1$; let $y =(y_1,..., y_{m_2}),\, \dim M_2=m_2$, be local coordinates on $M_2$ and denote by $(y, q) $ the corresponding coordinates on the cotangent bundle of $M_2$.

Recall that the space of polynomial in momenta integrals of a fixed degree of the geodesic flow is finite dimensional (under our $C^2$-smooth\-ness assumption, this was shown in \cite[Proof of Theorem~1]{KM}, while the $C^\infty$ case is older and is done already in \cite[Theorem~4.3]{Thompson}). We consider a basis $F_\a,\, \a=1, \dots, N$, for the space of the polynomial integrals of degree at most $d$ for the metric $g_1$. Similar to the proof of Lemma~\ref{lem:2}, for any geodesic $\gamma_2:\mathbb{I}\to M_2$ parameterised by an arc length $t$, we consider the totally geodesic immersion
\begin{equation*}
T:M_1 \times \mathbb{I} \to \overline{M} \ \textrm{ given by } \ (x,t)\mapsto (x, \gamma(t)).
\end{equation*}
The pullback of the integral $\overline{F}$ satisfied the assumptions of Lemma~\ref{lem:2} implying that
\begin{equation} \label{eq:8}
 \overline{F} = \sum_{\alpha=1}^N F_\alpha(x,p) Q_\alpha(y,q),
\end{equation}
where the functions $Q_{\alpha}$ do not depend on $x$ and on $p$ and are polynomial in $q$. 

We now show that the functions $Q_\alpha(y,q)$ are the integrals for the geodesic flow of $g_2$. Clearly, $\overline{\cH}=\cH_{1}+ \cH_{2}$, where $\cH_{i}, \, i=1,2$, is the Hamiltonian of the geodesic flow of $g_i$ on $M_i$. Note that $\cH_{1}$ does not depend on $(y,q)$ which implies
$\{\cH_{1}, Q_\alpha\}=0$. Similarly $\{\cH_{2}, F_\alpha\}=0$. Note also that $\{\cH_{1}, F_\alpha\}=0$, because $F_\alpha$ are integrals of the geodesic flows of $g_1$. Then for the integral given by~\eqref{eq:8} we have
\begin{equation*}
0=\{\cH_{1}+ \cH_{2}, \sum_{\alpha=1}^N F_\alpha Q_\alpha \} = \sum_{\alpha=1}^N \{\cH_{2}, Q_\alpha \} F_\alpha.
\end{equation*}
Since the integrals $F_\alpha$ are linearly independent, we obtain that all the brackets $\{\cH_{g_2}, Q_\alpha \}$ are zero, and the claim follows.
\end{proof}

Note that Corollary~\ref{cor:2} immediately implies Theorems~\ref{thm:2} and \ref{thm:3}. Indeed, if $M_1$ is compact then for every $y\in M_2$,
the subspace $\{(x\in M_1, \overline{p} \in U_{(x,y)}\overline{M} ) \}$ is compact, and so any smooth function on it is bounded, which proves Theorem~\ref{thm:2}. Similarly, if $M$ is compact, then the unit cotangent bundle $U\overline{M} $ is also compact, which gives Theorem~\ref{thm:3}.

\section{Proof of Theorem~\ref{t:decompo}}
\label{s:pf}

The following fact is well known (and is easy to check).

\begin{fact} \label{f:fact1}
{\ }
\begin{enumerate}[label=\emph{(\alph*)},ref=\alph*]
 \item \label{it:fnilp}
 If $f \in C^\infty(T^*M)$ is a homogeneous polynomial of degree $d$ in the momenta, 
 then $\{\cH,f\}$ is either zero, or a homogeneous polynomial of degree $d+1$ in the momenta.

 \item \label{it:dergeod}
 Let $(\gamma(s), p(s)), \, s \in (a,b) \subseteq \br$, be a trajectory of the geodesic flow of $g$. Then for $f \in C^\infty(T^*M)$ we have $\{\cH,f\}(\gamma(s),p(s))=\frac{d}{ds}(f(\gamma(s),p(s)))$.
\end{enumerate}
\end{fact}

We need the following lemma.

\begin{lemma} \label{l:findim}
 Let $(M,g)$ be a connected $C^\infty$ Riemannian manifold. In the notation of Section~\ref{s:intro}, every subspace $\sL^d_k(M) \subset C^\infty(T^*M), \; d \ge 0$, $k \ge 1$, is finite-dimensional.
\end{lemma}
\begin{proof}
 For $k=1$ we have $\sL^d_k(M)=\sK^d(M)$, and the claim follows from the fact that the space of Killing tensor fields of any rank on a connected Riemannian manifold is finite-dimensional~\cite[Theorem~4.3]{Thompson}. Suppose the claim is already established for some $k \ge 1$ and all $d \ge 0$. Then the subspace $\{\cH, C^\infty(T^*M) \} \cap \sL^{d+1}_k(M)$ is finite dimensional, and hence has a basis of the form $\{\{\cH, \phi_1\}, \dots , \{\cH, \phi_N\}\}$ for some $\phi_1, \dots, \phi_N \in C^\infty(T^*M)$. Then $\sL^d_{k+1}(M)=\sL^d_{k}(M) + \Span(\phi_1, \dots, \phi_N)$.
\end{proof}

\begin{proof}[Proof of Theorem~\ref{t:decompo}]
In the assumptions and notation of Theorem~\ref{t:decompo}, let $m_i=\dim M_i,\, i=1,2$, and let $(x^i, p_i), \, i=1, \dots, m_1$, be local coordinates on $T^*M_1$, with $p_i$ the momenta, and let $(y^j, q_j)\, j=1, \dots, m_2$, be local coordinates on $T^*M_2$, with $q_j$ the momenta.

Let $\overline{F}$ be an integral on $\overline{M}$ which is a homogeneous polynomial of degree $d$ in the momenta $(p_i, q_j)$. Consider the decomposition of $\overline{F}$ into the components $S_k(x,y,p,q), \, k=0, \dots, d$, each of which is homogeneous by both $p_i$ and $q_j$:
\begin{equation}\label{eq:KonMbar} 
 f(x,y,p,q)= \sum_{\ell=0}^d S_\ell(x,y,p,q),
\end{equation}
where $ \deg_p S_\ell(x,y,p,q)=\ell$ and $\deg_q S_\ell(x,y,p,q)=d-\ell$.

The Hamiltonian on $\overline{M}$ is given by $\overline{\cH} = \cH_1 + \cH_2$, where $\cH_i, \, i=1,2$, is the Hamiltonian on $M_i$. Separating the components homogeneous in $p_i$ and $q_j$ in the equation  $\{\cH_1+\cH_2,\overline{F}\}=0$ we obtain
 \begin{equation}\label{eq:hompq}
 \begin{array}{rcl}
 \{ \cH_2, S_0(x,y,p, q)\} & =& 0, \\
 \{ \cH_1, S_0(x,y,p, q)\} + \{\cH_2 , S_1(x,y,p, q)\} &=& 0, \\
 &\vdots& \\
 \{ \cH_1, S_{d-1}(x,y,p, q)\} + \{\cH_2, S_d(x,y,p,q)\} &=& 0, \\
 \{ \cH_1 , S_d(x,y,p, q)\}& =& 0
 \end{array}
 \end{equation}
(note that $S_0$ does not depend on $p$, and $S_d$ does not depend on $q$).

From the first equation of~\eqref{eq:hompq} we obtain that for every fixed $(x,p) \in T^*M_1$, the function $S_0(x,y,p, q)$ belongs to $\sL^d_1(M_2) $. As the space $\sL^d_1(M_2)$ is finite dimensional by Lemma~\ref{l:findim}, we get $S_0(x,y,p,q) \in (C^\infty(T^*M_1))^0 \otimes \sL^d_1(M_2)$, where we denote by $(C^\infty(T^*M_i))^k$ the space of smooth functions on $M_i, \, i=1,2$, which are homogeneous of degree $k$ in the momenta.

Acting on the second equation of~\eqref{eq:hompq} by $\{\cH_2, \ . \ \}$ and taking into account the first equation we obtain that for every fixed $(x,p) \in T^*M_1$, the function $S_1(x,y,p,q)$ belongs to $\sL^{d-1}_2(M_2)$, and hence Lemma~\ref{l:findim} implies that $S_1(x,y,p, q) \in (C^\infty(T^*M_1))^1 \otimes \sL^{d-1}_2(M_2)$.

Repeating this argument another $d-1$ times we obtain $S_k(x,y,p, q) \in (C^\infty(T^*M_1))^k \otimes \sL^{d-k}_{k+1}(M_2)$, for $k=0, \dots, d$. Interchanging the roles of $M_1$ and $M_2$ we get $S_k(x,y,p, q) \in \sL^k_{d-k+1}(M_1) \otimes (C^\infty(T^*M_2))^{d-k}$. As both spaces $\sL^{d-k}_{k+1}(M_2)$ and $\sL^k_{d-k+1}(M_1)$ are finite dimensional by Lemma~\ref{l:findim}, we obtain $S_k(x,y,p, q)
\in \sL^k_{d-k+1}(M_1) \otimes \sL^{d-k}_{k+1}(M_2)$, for $k=0, \dots, d$ (cf. the argument in the proof of Corollary~\ref{cor:2}).

Now by Fact~\ref{f:fact1}\eqref{it:fnilp}, the operator $H_1=\{\cH_1, \ . \ \}$ (defined by~\eqref{eq:H}) acting on the finite dimensional space $\oplus_{k=0}^d \sL^k_{d-k+1}(M_1)$ shifts the degree of every homogeneous component in $p$ by $1$, and hence is nilpotent. We can choose a basis for $\oplus_{k=0}^d \sL^k_{d-k+1}(M_1)$ consisting of elements homogeneous in $p$ relative to which the matrix of $H_1$ has a Jordan form. More precisely, we can decompose the space $\oplus_{k=0}^d \sL^k_{d-k+1}(M_1)$ into the direct sum of linear subspaces $W^\la$ such that $W^\la=\Span(f_1^\la, H^1_1 f_1^\la, \dots, H_1^{m_\la-1}f_1^\la)$, where $f_1^\la \in \sL^{k_\la}_{m_\la}(M_1)$, with $k_\la, m_\la \le d$. Similarly, we decompose $\oplus_{k=0}^{d-k} \sL^k_{k+1}(M_2)$ into the direct sum of linear subspaces $V^\mu$ such that $V^\mu=\Span(f_2^\mu, H_2^1 f_2^\mu, \dots, H_2^{m_\mu-1}f_2^\mu)$, where $f_2^\mu \in \sL^{k_\mu}_{m_\mu}(M_2)$, with $k_\mu, m_\mu$ $\le d$. We now have $f \in \oplus_{\la,\mu} W^\la \otimes V^\mu$, and for any $(\la, \mu)$, the space $W^\la \otimes V^\mu$ is invariant relative to both $H_1$ and $H_2$. The actions of $H_1$ on $W^\la$ can be identified with the action of the operator $\frac{d}{du}$ on the space of polynomial in $u$ of degree $m_{\la}-1$, with the element $f_1^\la$ corresponding to $u^m_{\la}-1$. Similarly, the actions of $H_2$ on $V^\mu$ can be identified with the action of the operator $\frac{d}{dv}$ on the space of polynomial in $v$ of degree $m_{\mu}-1$, with the element $f_2^\mu$ corresponding to $v^m_{\mu}-1$. Then the action of $H_1+H_2$ on $W^\la \otimes V^\mu$ corresponds to $\frac{d}{du} + \frac{d}{dv}$, and so its kernel is spanned by the powers of $u-v$ (up to and including the power $\min(m_\la-1,m_\mu-1)$), which correspond under our identification to the expressions given in~\eqref{eq:fM1M2}. This completes the proof of Theorem~\ref{t:decompo}.
\end{proof}

\section{An example of a complete product manifold with an irreducible Killing tensor}
\label{ex:noncomp}

Theorem~\ref{thm:2} tells that any Killing tensor on the Riemannian product of two manifolds one of which is compact is reducible. 
One may wonder whether the claim still holds if one relaxes the assumption of compactness to a weaker assumption of completeness.

Of course, one can easily construct irreducible Killing tensors on $(\mathbb{R}^n= \mathbb{R}^k\times \mathbb{R}^{n-k},g^n_{euclidean}= g^k_{euclidean} +g^{n-k}_{euclidean} ) $. For example, on $\mathbb{R}^2$ with the Cartesian coordinates $(x,y)$, the Killing vector field $y\tfrac{\partial }{\partial x}- x\tfrac{\partial }{\partial y})$ is irreducible.

The goal of this section is to give an example of a complete Riemannian product manifold, whose factors contain no locally reducible open domains, which admits a Killing tensor not being the sum of the symmetrised tensor products of Killing tensors on the factors.

We start with a metric on $\br^3$ given, relative to the cylindrical coordinates $(r,\theta,z), \; r \ge 0$, $\theta \in [0,2\pi), z \in \br$, by
 \begin{equation}\label{eq:ex}
 ds^2=dr^2 + \frac{r^2}{1+r^2} \, d\theta^2 + \frac{1}{1+r^2} \, dz^2.
 \end{equation}
Note that for small $r>0$, we have $ds^2=dr^2 + (r^2+o(r^4)) d\theta^2 + (1-r^2+o(r^4)) dz^2$, so that the metric~\eqref{eq:ex} is well defined and smooth at $r = 0$.

The vector fields $\tfrac{\partial}{ \partial\theta}$ and $\tfrac{\partial }{\partial z}$ are clearly Killing. The dual covectors are $\omega_1=\frac{r^2}{1+r^2} d\theta$ and $\omega_2=\frac{1}{1+r^2} dz$, respectively. Consider the covector field $\Omega=rdr + \frac{2z}{1+r^2}dz$ (note that it is well-defined and smooth relative to our cylindrical coordinates). The symmetrised covariant derivative of $\Omega$ is the covariant tensor of type $(0,2)$ given by $\nabla \Omega= dr^2 +\frac{r^2}{(r^2 + 1)^2}d\theta^2 + \frac{r^2 + 2}{(r^2 + 1)^2}dz^2$. This tensor is Killing, as one easily sees that $\nabla \Omega = ds^2 - \omega_1^2 + \omega_2^2$ (note that $\Omega$ by itself is not Killing, as $\nabla \Omega \ne 0$).

A direct computation shows that at a generic point $x \in \br^3$, there is no nonzero vector $X \in T_x \br^3$ such that the sectional curvature is zero along all the two-planes containing $X$. This shows that the metric $ds^2$ given by~\eqref{eq:ex} is not a product metric, even locally. To see that $ds^2$ is complete we consider a $C^1$ ray $c(t) = (r(t),\theta(t),z(t)), t \ge 0$, which leaves any compact subset of $\br^3$. If $r(t)$ is unbounded, then it clearly has infinite length, as $\int_{0}^{t} \|c'(t)\|dt \ge |r(t)-r(0)|$. If $r(t)$ is bounded by some constant $C > 0$, then $z(t)$ must be unbounded, and $c(t)$ also has infinite length, as $\int_{0}^{t} \|c'(t)\|dt \ge (1+C^2)^{-1}|z(t)-z(0)|$.

We now take the product of the two copies of the metric~\eqref{eq:ex}: $dS^2 = dr_1^2 + \frac{r_1^2}{1+r_1^2} \, d\theta_1^2 + \frac{1}{1+r_1^2} \, dz_1^2+dr_2^2 + \frac{r_2^2}{1+r_2^2} \, d\theta_2^2 + \frac{1}{1+r_2^2} \, dz_2^2$, the covector fields $\Omega_1=r_1dr_1 + \frac{2z_1}{1+r_1^2}dz_1$ and $\Omega_2=r_2dr_2 + \frac{2z_2}{1+r_2^2}dz_2$, and their symmetrised covariant derivatives $\nabla \Omega_1$ and $\nabla \Omega_2$. Then by the construction in formula~\eqref{eq:fM1M2} (or by a direct calculation) one sees that the $(0,3)$ tensor $K=\Omega_1 \odot \nabla \Omega_2 - \Omega_2 \odot \nabla \Omega_1$ is Killing for the product metric $dS^2$. The fact that $K$ is not the sum of the products of the Killing tensors on the factors follows, as the $(0,1)$ tensors $\Omega_1$ and $\Omega_2$ are not Killing and therefore the Poisson bracket of the corresponding integral with $\cH_1$ is not zero.

\subsection*{Data availability and conflict of interest statements} All data generated or analysed during this study are included in this published article. On behalf of all authors, the corresponding author states that there is no conflict of interest.

\printbibliography

\end{document}